\documentclass[12pt,reqno]{amsart}
\advance \topmargin by -\headheight
\advance \topmargin by -\headsep
\evensidemargin \oddsidemargin
\usepackage{amsmath}
\usepackage{amsfonts}
\usepackage{stmaryrd}
\usepackage{amssymb}
\usepackage{amsthm}
\usepackage{csquotes}
\usepackage{color}
\usepackage{mathrsfs}
\usepackage{dsfont}
\usepackage{bm}
\usepackage{cite}
\usepackage{soul}

\numberwithin{equation}{section}
\renewcommand\vec{\bm}

\newtheorem{theorem}{Theorem}[section]

\newtheorem{lemma}[theorem]{Lemma}
\newtheorem{Proposition}[theorem]{Proposition}

\newtheorem{Question}[theorem]{Question}

\usepackage{mathtools}

\DeclarePairedDelimiter{\floor}{\lfloor}{\rfloor}

\title{Finding large additive and multiplicative Sidon sets in sets of integers}

\author{Yifan Jing}
\address{Mathematical Institute, University of Oxford, Oxford OX2 6GG, UK}
\email{jing@maths.ox.ac.uk}

\author{Akshat Mudgal}
\address{Mathematical Institute, University of Oxford, Oxford OX2 6GG, UK}
\email{mudgal@maths.ox.ac.uk}

\thanks{YJ and AM are supported by Ben Green’s Simons Investigator Grant, ID:376201. }


\renewcommand\vec{\bm}

\begin{document}

\begin{abstract}
Given $h,g \in \mathbb{N}$, we write a set $X \subset \mathbb{Z}$ to be a $B_{h}^{+}[g]$ set if for any $n \in \mathbb{Z}$, the number of solutions to the additive equation
$n = x_1 + \dots + x_h$
with $x_1, \dots, x_h \in X$ is at most $g$, where we consider two such solutions to be the same if they differ only in the ordering of the summands. We define a multiplicative $B_{h}^{\times}[g]$ set analogously. In this paper, we prove, amongst other results, that there exist absolute constants $g \in \mathbb{N}$ and $\delta>0$ such that for any $h \in \mathbb{N}$ and for any finite set $A$ of integers, the largest $B_{h}^{+}[g]$ set $B$ inside $A$ and the largest $B_{h}^{\times}[g]$ set $C$ inside $A$ satisfy
\[ \max \{ |B| , |C| \} \gg_{h} |A|^{(1+ \delta)/h }. \]
In fact, when $h=2$, we may set $g = 31$, and when $h$ is sufficiently large, we may set $g = 1$ and $\delta \gg (\log \log h)^{1/2 - o(1)}$. The former makes progress towards a recent conjecture of Klurman--Pohoata and quantitatively strengthens previous work of Shkredov.
\end{abstract}

\maketitle

\section{Introduction}

In this paper, we investigate a question lying at the interface of two well-known topics in additive and combinatorial number theory; these being the study of \emph{Sidon sets} and the Erd\H{o}s--Szemer\'{e}di sum-product phenomenon. We begin by introducing the former, and thus, given natural numbers $g,h$, we write a finite set $A$ of natural numbers to be a $B^{+}_h[g]$ set if for any $n \in \mathbb{N}$, the number of distinct solutions to the equation
\begin{equation} \label{add4}
    n = a_1 + \dots + a_h
\end{equation} 
with $a_1, \dots, a_h \in A$ is at most $g$. Here, we consider two such solutions to be the same if they differ only in the ordering of the summands. We define a $B^{\times}_{h}[g]$ set similarly by replacing the additive equation $\eqref{add4}$ with its multiplicative analogue 
$n = a_1 \dots a_h$. When $h=2$ and $g=1$, we refer to $B^+_{h}[g]$ and $B^{\times}_h[g]$ sets as Sidon and multiplicative Sidon sets respectively. 
\par

Sidon sets and their many generalisations have been analysed in various works from many different perspectives, with a classical problem concerning the size of the largest $B_{h}^{+}[g]$ subset of $[N] := \{1, \dots, N\}$. Writing $B$ to be the largest $B_h^{+}[1]$ subset of $[N]$, it is known from work of Singer \cite{Si1938} and Bose--Chowla \cite{BC1963} that $|B| > N^{1/h}$ for infinitely many choices of $N$. As for upper bounds, when $h=2$, Erd\H{o}s and Tur\'{a}n \cite{ET1941} proved that 
\[ |B| < N^{1/2} + N^{1/4} + 1.\]
While there were subsequent improvements to the constant term in the upper bound, it was only very recently that the $N^{1/4}$ term was improved to $0.998N^{1/4}$ for all sufficiently large $N$ by Balogh, F\"{u}redi and Roy \cite{BFS2021}. The case when $h \geq 3$ appears to be significantly harder. Indeed, there has been a rich history of work surrounding this problem, see, for instance, \cite{Ci2001} and the references therein, and the best known upper bound follows from work of Green \cite{Gr2001}, who showed that
\begin{equation} \label{erds2}
 |B| \leq ( (\floor{h/2}!)^2 \sqrt{\pi h/2}  (1 + \epsilon(h)) )^{1/h} N^{1/h} ,  
\end{equation}
where $\epsilon(h) \to 0$ as $h \to \infty$. Similar upper and lower bounds are known for the size of the largest $B_h^{+}[g]$ subset of $[N]$, and we refer the reader to \cite{JTT2021} for the best known estimates towards this problem. Thus, when $h \geq 3$, there is a large gap between the known upper and lower bounds for the size of the largest $B_h^{+}[1]$ set in $[N]$, and closing this gap is a major open problem in the area, see for example \cite[Problem $11.14$]{Go2000}.

On the other hand, finding large $B_{h}^{\times}[1]$ sets in $[N]$ is relatively elementary, since the set of prime numbers smaller than $N$ form such a set. The fact that this example is essentially sharp was shown by Erd\H{o}s \cite{Er1969}, who proved that the largest multiplicative Sidon set $C \subset [N]$ satisfies
\[ N^{3/4} (\log N)^{-3/2} \ll |C| - \pi(N) \ll N^{3/4} (\log N)^{-3/2}, \]
where $\pi(N)$ counts the number of prime numbers in $[N]$. The size of large $B_{h}^{\times}[1]$ sets have also been investigated for $h \geq 3$ and are known to be close to $\pi(N)$, see \cite{Pa2015} for more details.
\par

Hence, while the largest $B_{h}^{+}[1]$ sets in $[N]$ are hard to characterise and have size close to $N^{1/h}$, one can find substantially large $B_{h}^{\times}[1]$ sets in $[N]$ which have size $(1 + o(1))N/\log N$. This discrepancy can be justified by noting that the set $[N]$ exhibits large amounts of additive structure but has relatively low multiplicative structure. Similarly, if one writes $B$ to be the largest $B_h^{+}[1]$ subset and $C$ to be the largest $B_h^{\times}[1]$ subset of $\{2,4, \dots, 2^N\}$, one can see that $|B| \gg_h N$ while $N^{1/h} \leq |C| \ll_h N^{1/h}$. This is reminiscent of a well-known conjecture of Erd\H{o}s and Szemer\'{e}di \cite{ES1983} which roughly states that large amounts of additive and multiplicative structure cannot simultaneously coexist in a given finite set $A \subset \mathbb{Z}$. Note that for any $A \subset \mathbb{Z}$ with $|A| =N$, writing $B$ and $C$ to be the largest $B_h^{+}[1]$ and $B_h^{\times}[1]$ subsets of $A$, one has $|B|,|C| \gg_h N^{1/h}$ due to a classical result of  Koml\'{o}s--Sulyok--Szemer\'{e}di \cite{KSS1975}. On the other hand, motivated by the sum-product philosophy and the above examples, one might expect that $\max\{|B|, |C|\} \gg_h N^{(1+ \delta_h)/h}$, for some $\delta_h >0$.
\par

This heuristic was exemplified in a recent problem posed by Klurman and Pohoata \cite{Po2021}, who conjectured that in the above setting, whenever $h=2$, one must have
\begin{equation} \label{cocon}
\max \{ |B|, |C| \} \gg_{\delta} |A|^{1/2 + \delta}, 
\end{equation} 
for each $\delta \in (0,1/2)$. Our main result in this paper confirms an approximate version of the preceding heuristic for all $h \in \mathbb{N}$. 

\begin{theorem} \label{main}
There exist absolute constants $g \in \mathbb{N}$ and $\delta>0$ such that for every $h \in \mathbb{N}$ and for every finite set $A \subset \mathbb{Z}$, writing $B$ and $C$ to be the largest $B_{h}^+[g]$ and $B^{\times}_h[g]$ subsets of $A$ respectively, one has
\[ \max\{|B|, |C|\} \gg_h |A|^{(1 + \delta)/h}. \]
\end{theorem}

Theorem \ref{main} provides non-trivial sum-product estimates in a straightforward manner. Indeed, given $h \in \mathbb{N}$ and a finite set $X \subset \mathbb{R}$, we define its $h$-fold sumset $hX$ and $h$-fold product set $X^{(h)}$ as
\[ hX = \{ x_1 + \dots + x_h : x_1, \dots, x_h \in X\} \ \ \text{and} \ \ X^{(h)} = \{x_1 \dots x_h : x_1, \dots, x_h \in X\} . \]
For any finite set $A \subset \mathbb{Z}$, one has the trivial estimates $|A| \leq |hA|, |A^{(h)}| \leq |A|^h$. The aforementioned Erd\H{o}s--Szemer\'{e}di sum-product conjecture asserts that for any $h \in \mathbb{N}$ and any $\varepsilon >0$ and any $A \subset \mathbb{Z}$, one has 
\begin{equation} \label{j1}
     \max\{ |hA|, |A^{(h)}| \} \gg_{h, \varepsilon} |A|^{h - \varepsilon}.  
     \end{equation}
Theorem \ref{main} implies that for any finite $A \subset \mathbb{Z}$, one has
\begin{equation} \label{bc45}
\max\{|hA| , |A^{(h)}|\} \geq \max\{|hB| ,|C^{(h)}|\} \gg_{g,h} |B|^h + |C|^h \gg_{h}   |A|^{1 + \delta}, 
\end{equation}
thus providing non-trivial results towards this problem.

It appears to be the case that the problem of obtaining large $B_h^{+}[1]$ and $B_h^{\times}[1]$ sets inside arbitrary sets seems harder and different when compared to the sum-product conjecture. Firstly, noting $\eqref{j1}$, one might expect every finite set $A \subset \mathbb{Z}$ to contain either a $B_{h}^+[1]$ set or a $B_{h}^{\times}[1]$ set of size close to $|A|^{1 - o(1)}$, but this does not hold true for any $h \geq 2$. Indeed, various elementary constructions of Green--Peluse (unpublished), Roche-Newton--Warren \cite{RNW2021} and Shkredov \cite{Sh2021} imply that \eqref{cocon} can not hold for any $\delta > 1/6$. Utilising some graph theoretic results from \cite{NV05}, we generalise their constructions for every $h \geq 3$.

\begin{Proposition}\label{prop: construction}
Let $h \geq 2$ and $N$ be natural numbers. Then there exists a set $A \subset \mathbb{N}$ with $|A| \gg N$ such that the largest $B_h^+[1]$ subset $B$ and the largest $B_h^\times[1]$ subset $C$ of $A$ satisfy
\[
\max\{|B|, |C|\}\ll_h \begin{cases}
|A|^{\frac{1}{2}+\frac{1}{2h+2}}\quad&\text{ when }h\text{ is even},\\
|A|^{\frac{1}{2}+\frac{1}{2h}}\quad&\text{ when }h\text{ is odd}.
\end{cases}
\]
\end{Proposition}

Secondly, we observe that a non-trivial sum-product estimate of the form \eqref{j1} for any fixed $h_0 \geq 2$ and $\delta_0 >0$ implies a non-trivial sum-product result for every $h \geq 2$ with $\delta = \delta_0/h_0$. On the other hand, even the most potentially optimal version of Theorem \ref{main} in the $h=2$ case, with, say, $g=1$ and $\delta\leq1/6$, does not seem to deliver any non-trivial versions of Theorem \ref{main} for any $h\geq3$ in a straightforward manner, see end of \S2 for further details. In fact, this is one of the reasons why we need to develop quite different approaches for different values of $h$. Indeed, we will prove Theorem \ref{main} via a combination of Theorems \ref{th3}, \ref{th2} and \ref{th1}, which deal with the cases when $h > C$ for some large absolute constant $C>0$, when $3 \leq h \leq C$, and when $h=2$ respectively.

We briefly mention the best known results towards the sum--product conjecture. When $h=2$, Rudnev--Stevens \cite{RS2022} proved that $\max\{|2A|,|A^{(2)}|\} \gg_c |A|^{4/3 + c}$ for any $c< 2/1167$, while for large values of $h$, one has a famous result of Bourgain--Chang \cite{BC2005} which states that
\begin{equation} \label{bces4}
 \max\{ |hA|, |A^{(h)}| \} \gg |A|^{c' (\log h)^{1/4}}     
\end{equation} 
for some $c'>0$. This exponent has since only been improved once by P\'{a}lv\"{o}lgyi--Zhelezov \cite{PZ2021} to $(\log h)^{1 - o(1)}$. Noting these estimates, it is natural to ask whether one can allow $\delta \to \infty$ as $h \to \infty$ in Theorem \ref{main}, and our next result confirms this in the affirmative.

\begin{theorem} \label{th3}
Let $h$ be a natural number, let $A \subset \mathbb{Z}$ be a finite set, and let $B$ and $C$ be the largest $B_{h}^{+}[1]$ and $B_h^{\times}[1]$ sets in $A$ respectively. Then 
\[ \max \{ |B|, |C| \} \gg |A|^{\frac{\eta_h}{h}} ,\]
where $\eta_h \gg (\log \log h)^{1/2 - o(1)}$.
\end{theorem}

Theorem $\ref{th3}$ delivers sum-product estimates akin to those of Bourgain--Chang  in a straightforward manner by exploiting \eqref{bc45}. In fact, a key ingredient in the proof of Theorem $\ref{th3}$ entails amalgamating probabilistic techniques with another generalisation of the sum-product estimate of Bourgain--Chang proved by the second author in \cite{Mu2021d} involving the so-called \emph{low-energy decompositions}, see Lemma $\ref{mu1}$. 
Moreover, while Theorem $\ref{th3}$ provides a much stronger conclusion than Theorem \ref{main} when $h$ is sufficiently large, the former does not give any non-trivial conclusions for potentially smaller values of $h$. In fact, it is precisely these cases that requires most of our focus and make up the core of our argument. We begin by further subdividing this setting into two subcases, the first of these being when $h \geq 3$. 

\begin{theorem} \label{th2}
Let $A \subset \mathbb{Z}$ be a finite set and let $h \geq 3$. Then there exists $g \leq 30 h$ and $\delta_h \gg h^{-3}$ such that the largest $B_h^{+}[g]$ set $B$ inside $A$ and the largest $B_h^{\times}[g]$ set $C$ inside $A$ satisfy
\[ \max \{ |B| , |C| \} \gg |A|^{1/h + \delta_h} . \]
\end{theorem}

As previously mentioned, in contrast to Theorem \ref{th3}, we are unable to employ Lemma \ref{mu1} in the proof of Theorem \ref{th2} since Lemma \ref{mu1} only provides non-trivial estimates for sufficiently large values of $h$. Thus, we begin by combining sum-product results from \cite{So2009} along with Balog-Szemer\'{e}di-Gowers type arguments from \cite{Sch2015} to obtain a dichotomy where we either have
\begin{equation}  \label{pl3}
     r_h(A;n) =   | \{(a_1, \dots, a_h) \in A^h : a_1  + \dots + a_h = n\} |  \ll |A|^{h- 1- \nu}
\end{equation}
for every $n \in hA$, for some small constant $\nu>0$, or $A$ has a large subset with a small $h$-fold multiplicative energy. Since we eventually need to obtain a large $B_h^{+}[g]$ or $B_h^{\times}[g]$ set inside $A$, we need to amplify the bound in \eqref{pl3} by considering higher moments of the form 
\[ E_{h,k}(A) = \sum_{n\in hA} r_h(A;n)^k.\] 
These count the number of solutions to the system of simultaneous equations
\begin{equation} \label{jk3}
a_1 + \dots + a_h = a_{h+1} + \dots + a_{2h} = \dots = a_{h(k-1) + 1} + \dots + a_{hk} ,
\end{equation} 
with $a_1, \dots, a_{hk} \in A$. With suitable estimates on $E_{h,k}(A)$ in hand, we now proceed with the key step of our proof, whereupon, given $2 \leq l \leq hk-1$, we are necessitated to find strong upper bounds for the number of solutions to \eqref{jk3} when we have precisely $l$ distinct elements $a_1, \dots, a_{hk}$. This requires an amalgamation of a variety of combinatorial and linear algebraic ideas, and  forms the content of Lemma \ref{lim2}. Having proven this, we are now able to apply probabilistic techniques to extract either a large large $B_h^{+}[g]$ set or a large $B_h^{\times}[g]$ set inside $A$, for some $g \ll h$.
We further remark that we have not carefully optimised the dependency of $g$ and $\delta_h$ on $h$ in the above result, since our main aim has been to show that $g \ll h$ and $\delta_h >0 $.

It is worth remarking that in the other subcase, that is, when $h=2$, the above strategy fails at the first step since good lower bounds for $r_2(A;n)$, for some $n \in 2A$, can not be employed to extract a large subset of $A$ with small multiplicative energy. Indeed, for any finite set $X \subset \mathbb{Z}$ with $|X|=K$, one may set 
\[ A = \cup_{x \in X}  (x - A') \cup A' \ \  \text{where} \ \ A'= \{10^n : 1 \leq n \leq N\}.\]
Here, we have  $r_2(A;x) \gg |A|/K$ for every $x \in X$, while for any $B \subseteq A$ with $|B| = \alpha |A|$, one may apply Cauchy's inequality to deduce that
\[ M_{2,2}(B) := |\{(b_1,b_2,b_3,b_4) \in B^4 : b_1 b_2 = b_3 b_4\}|  \geq \frac{|B|^4}{|A^{(2)}|} \gg \alpha |B|^3  K^{-2}.  \]
Nevertheless, we record the following version of Theorem \ref{main} in the $h=2$ case, which we prove with yet another different set of techniques.

\begin{theorem} \label{th1}
Let $A \subset \mathbb{Z}$ be a finite set. Then there exist $g \leq 31$ and $\delta >0$ such that the largest $B_2^{+}[g]$ set $B$ inside $A$ and the largest $B_2^{\times}[g]$ set $C$ inside $A$ satisfy
\[ \max \{ |B|, |C| \} \gg |A|^{1/2  + \delta} . \]
\end{theorem}

We briefly mention some of the ideas involved in the proof of Theorem \ref{th1}. Here, we consider a subset $S \subseteq 2A$ such that $r_2(A;n) \geq |A|^{1 - \varepsilon}$ for all $n \in S$, for some suitable choice of $\varepsilon>0$. If $|S|<|A|^{\eta}$, for some fixed $1/3<\eta<1/2$, then we can procure good upper bounds $E_{2,k}(A)$ for reasonable values of $k$ and consequently obtain a large $B_2^{+}[g]$ set inside $A$. On the other hand, if $|S|\geq |A|^{\eta}$, then we may find a large subset $A'$ of $A$ such that for every $a' \in A'$, the set $(A+a') \cap S'$ is large, where $S' \subseteq S$ satisfies $|S'| = |A|^{\eta}$. Our aim now is show that for any $n \neq 0$, one obtains some power saving over the trivial bound for the number of solutions to $n = a_1' a_2'$ with $a_1',a_2' \in A'$, since this can then be amplified by considering the multiplicative analogue of $E_{2,k}(A)$, which, in turn, would deliver a large $B_2^{\times}[g]$ sets inside $A$ via probabilistic techniques. Moreover, since the set $(A+a') \cap S'$ is large for every $a' \in A'$, our goal can be rephrased in terms of obtaining good upper bounds for the number of solutions to
\[ (x_1 - y_1)(x_2 - y_2) = 1 \]
with $x_1,x_2$ and $y_1, y_2$ lying in suitable dilates of $A$ and $S'$ respectively. This is precisely the key step in our proof of Theorem \ref{th1}, where we employ combinatorial geometric techniques to analyse the number of incidences between a large point set and somewhat few translates of a hyperbola. We record this estimate in the form of Theorem \ref{hyp}, which can be interpreted as a quantitative Euclidean analogue of a result of Bourgain \cite{Bo2012} and would perhaps be of independent interest.

Theorem \ref{th1} quantifies the only other affirmative result in this direction, the latter being due to Shkredov \cite{Sh2021}, who showed that Theorem $\ref{th1}$ holds for some $g \leq K$, for some potentially very large value of $K>0$. Moreover, the reader may observe that Theorems $\ref{th3}, \ref{th2}$ and $\ref{th1}$ combine to deliver Theorem \ref{main}.


Considering Proposition \ref{prop: construction}, we see that even for large values of $h$, there is a big gap between the lower bounds that are provided by Theorem $\ref{th3}$ and the upper bounds presented in Proposition $\ref{prop: construction}$. This naturally leads one to the following question.

\begin{Question}
For each $h \in \mathbb{N}$, let $\Lambda_h$ be the supremum of all real numbers $\eta_h >0$ which satisfy the following statement. Any finite set $A$ of natural numbers contains either a $B_{h}^+[1]$ set or a $B_{h}^{\times}[1]$ of size at least $C_{h,\eta_h} |A|^{\eta_h}$, for some absolute constant $C_{h,\eta_h}>0$. Find $\Lambda_h$.
\end{Question}

In particular, Theorem $\ref{th3}$ and Proposition $\ref{prop: construction}$ combine to imply the bound 
\[ h^{-1}(\log \log h)^{1/2 - o(1)} \ll \Lambda_h \leq 1/2 + 1/(2h+2) \]
for even values of $h$, and it would be interesting to know whether $\Lambda_h \to 0$ as $h \to \infty$. Even for specific choices of sets $A$, this question exhibits close connections to deep conjectures in extremal graph theory concerning bipartite versions of the Tur\'{a}n function for forbidden even cycles, see \S4 for further details.
\par

We finish this section by providing a brief outline of our paper. We use \S2 to present some further remarks surrounding the problem of finding large additive and multiplicative Sidon sets in arbitrary sets of integers. 
In \S3, we state various preliminary definitions and lemmata that we will frequently use throughout our paper.
We utilise \S4 to present two proofs of Proposition \ref{prop: construction} using graph-theoretic ideas.  Next, we employ \S5 to study estimates on number of solutions to systems of equations with repetitive variables, which we will then use along with probabilistic methods in \S6 to prove that sets with low additive or multiplicative energies contain large additive or multiplicative Sidon sets. In \S7, we prove some incidence estimates that we require in the proof of Theorem $\ref{th1}$. We conclude our paper by recording the proofs of Theorems \ref{th3}, \ref{th2} and \ref{th1} in \S8.

\textbf{Notation.} In this paper, we use Vinogradov notation, that is, we write $X \gg_{z} Y$, or equivalently $Y \ll_{z} X$, to mean $|X| \geq C_{z} |Y|$ where $C$ is some positive constant depending on the parameter $z$.  We use $e(\theta)$ to denote $e^{2\pi i \theta}$ for every $\theta \in \mathbb{R}$. Moreover, for every natural number $k$ and for every non-empty, finite set $Z$, we use $|Z|$ to denote the cardinality of $Z$, and we write $Z^k = \{ (z_1, \dots, z_k)  \ |  \ z_1, \dots, z_k \in Z\}$. For every natural number $n \geq 2$, we use boldface to denote vectors $\vec{x} = (x_1, x_2, \dots, x_n) \in \mathbb{R}^n$ and we write $\vec{x}^T$ to be the transpose of $\vec{x}$.

\textbf{Acknowledgements.} 
The authors would like to thank Ben Green for pointing to this problem as well as for various helpful discussions. The second author would like to thank David Ellis and Misha Rudnev for useful comments. The authors are grateful to the anonymous referee for many helpful comments.


\section{Further remarks}

We employ this section to highlight the various connections and differences between the problem of finding large additive and multiplicative Sidon sets and some of the other modern results towards the sum-product conjecture. We begin by defining the so-called \emph{low energy decompositions}, a concept which was introduced first by Balog--Wooley \cite{BW2017}. Thus, for any finite set $A \subset \mathbb{R}$ and any $s \in \mathbb{N}$, we write
\[ E_{s,2}(A) = \sum_{a_1, \dots,a_{2s} \in A} \mathds{1}_{a_1 + \dots + a_s = a_{s+1} + \dots + a_{2s}} \ \ \text{and}  \ \ M_{s,2}(A) = \sum_{a_1, \dots,a_{2s} \in A} \mathds{1}_{a_1  \dots  a_s = a_{s+1}  \dots a_{2s}}.  \]
By a simple application of Cauchy's inequality, one sees that
\begin{equation} \label{csstandard}
 |sA| E_{s,2}(A) \geq |A|^{2s} \ \ \text{and} \ \ |A^{(s)}| M_{s,2}(A) \geq |A|^{2s}. \end{equation}
Hence, generalising sum-product type inequalities, Balog--Wooley showed that any finite $A \subset \mathbb{Z}$ may be written as $A = X \cup Y$ with $X,Y$ disjoint such that
\[ E_{s,2}(X) , M_{s,2}(Y) \ll |A|^{2s - 1 - \delta}, \]
with $\delta = 2/33 - o(1)$. Noting \eqref{bces4}, they further speculated that one should be able to have $\delta \to \infty$ as $s \to \infty$. This was confirmed by the second author \cite{Mu2021} in the following quantitative form.

\begin{lemma} \label{mu1}
Let $s$ be a natural number and let $A$ be a finite set of integers. Then $A$ may be written as $A = X \cup Y$ for disjoint sets $X,Y$ such that
\[ E_{s,2}(X) \ll_{s} |X|^{2s - \sigma_s} \ \text{and} \ M_{s,2}(Y) \ll_{s} |Y|^{2s - \sigma_s}, \]
where $\sigma_s \geq D (\log \log s)^{1/2} (\log \log \log s)^{-1/2},$ for some absolute constant $D >0$. 
\end{lemma}

Firstly, this delivers sum-product estimates akin to the work of Bourgain--Chang \cite{BC2005} in a straightforward manner. In particular, applying \eqref{csstandard}, we find that
\begin{align*}
    |sA| + |A^{(s)}| & \geq |sX| + |Y^{(s)}| \geq |X|^{2s}E_{s,2}(X)^{-1} + |Y|^{2s} M_{s,2}(Y)^{-1} \gg_{s} |X|^{\sigma_s} + |Y|^{\sigma_s} \gg_{s} |A|^{\sigma_s},
\end{align*}
where $\sigma_s \gg (\log \log s)^{1/2 - o(1)}$. Moreover, we will later combine this with probabilistic methods, see Lemmata \ref{gens} and \ref{mu2}, to obtain large $B_s^{+}[1]$ and $B_s^{\times}[1]$ sets inside arbitrary sets, thus proving Theorem \ref{th3}. Upon optimising our random sampling and deletion argument, as in the above lemmata, one can show that whenever $M_{s,2}(Y) \ll_s |Y|^{2s - \sigma_s}$ for any $s\geq 2$ and $\sigma_s >0$, there exists a large $B_s^{\times}[1]$ set $C \subseteq Y$ with $|C| \gg_s |Y|^{\frac{\sigma_s}{2s-1}}$. One might wonder whether this can be improved to $|C| \gg_s |Y|^{\frac{\sigma_s}{s}}$, since in the cases when $Y = \{2,4,\dots, 2^N\}$ and $Y = [N]$, the latter estimate exhibits the correct order. Perhaps surprisingly, one can show that the random sampling and deletion argument is sharp when $s=2$, by considering the set $A$ described in \eqref{sur67}. Indeed, this set $A$ satisfies $M_{2,2}(A) \ll |A|^{2}$ while the largest $B_2^{\times}[1]$ subset $C$ of $A$ satisfies $|C| \ll |A|^{2/3}$. On the other hand, the question concerning the necessity of the quantitative loss in exponent that occurs via random sampling and deletion type arguments when $s \geq 3$ seems like a hard problem, which is very closely related to finding sharp estimates for the Tur\'{a}n function for forbidden even cycles in bipartite graphs. 


The converse direction also holds, that is, one can use results about Sidon sets to obtain low-energy decompositions, and here, there seems to be no such losses. For instance, via elementary arguments as in \cite[Proposition 9.2]{Mu2023}, one can show that Theorem \ref{th3} implies that any finite set $A \subset \mathbb{Z}$ may be written as $A = X \cup Y$, with $X,Y$ disjoint, such that
\[ E_{s,2}(X), M_{s,2}(Y) \ll_s |A|^{2s - \eta_s}, \]
This deduction is quantitatively sharp, as can be checked by noting the case when $A = [N]$. 
\par

We finally explain, in further detail, the remark following Proposition \ref{prop: construction}. Suppose we have a result that implies that for some fixed integer $h_0 \geq 2$ and some fixed $\delta_0 >0$, one has $\max\{|h_0 A|, |A^{(h_0)}|\} \gg |A|^{1 + \delta_0}$ for every finite set $A \subset \mathbb{Z}$. Then by employing Lemma \ref{prineq}, we see that
\begin{align*}
 \max\{|2A|, |A^{(2)}| \}  \geq |A| \max \{ (|h_0 A| / |A|)^{1/h_0}, (|A^{h_0}| / |A|)^{1/h_0} \}   \gg |A|^{1 + \delta_0/h_0},
\end{align*} 
whence, 
\[ \max\{|sA|, |A^{(s)}| \} \geq    \max\{|2A|, |A^{(2)}| \}  \geq    |A|^{1 + \delta_0/h_0}  \]
holds for every $s \geq 2$ and every finite $A \subset \mathbb{Z}$. On the other hand, there does not seem to be an analogue of the above manoeuvre in the setting of finding non-trivially large $B_h^{+}[1]$ and $B_h^{\times}[1]$ sets. Indeed, suppose that one could prove the most optimal version of Theorem \ref{th1} with $g=1$ and $\delta = 1/6$. Applying \cite[Proposition 9.2]{Mu2023}, we now have that any finite set $A \subset \mathbb{Z}$ may be written as $A = X \cup Y$, with $X,Y$ disjoint, such that
$E_{2,2}(X), M_{2,2}(Y) \ll |A|^{3 - 1/3}.$ Applying Lemma \ref{awk}, we get that $E_{s,2}(X), M_{s,2}(Y) \ll |A|^{2s - 1 - 1/3}$ for every $s \geq 2$. These estimates, when combined with the above random sampling and deletion type arguments, will provide either a $B_s^{+}[1]$ set or a $B_s^{\times}[1]$ set inside $A$ of size around $|A|^{\frac{4}{6s - 3}}$, which is significantly smaller than the ``trivial" lower bound $|A|^{\frac{1}{s}}$ that one gets from applying results from \cite{KSS1975}.


\section{Preliminary Lemmata}

Let $s,k$ be natural numbers and let $A$ be some finite, non-empty set of real numbers. For each $n \in \mathbb{R}$, we denote
\[ r_{s}(A;n) = \{ (a_1, \dots, a_s) \in A^s \ | \ a_1 + \dots + a_s = n\} \]
and
 \[ m_{s}(A;n) = \{ (a_1, \dots, a_s) \in A^s \ | \ a_1  \dots  a_s = n\} . \]
These have a natural connection to counting solutions to additive and multiplicative equations, and in particular, writing $E_{s,k}(A) = \sum_{n \in sA} r_{s}(A;n)^k,$ we see that $E_{s,k}(A)$ counts the number of solutions to the system of equations
\[ a_1 + \dots + a_s = a_{s+1} + \dots + a_{2s} = \dots = a_{(k-1)s + 1} + \dots + a_{ks}, \]
with $a_1, \dots, a_{ks} \in A$. Similarly, we define $M_{s,k}(A) = \sum_{n \in A^{(s)}} m_{s}(A;n)^k,$ wherein, we note that $M_{s,k}(A)$ counts the number of solutions to the system of equations
\[ a_1 \dots a_s =  a_{s+1}  \dots  a_{2s} = \dots = a_{(k-1)s + 1}  \dots  a_{ks}, \]
with $a_1, \dots, a_{ks} \in A$. 
\par

It is worth noting some straightforward properties of the representation function $r_{s}(A; \cdot)$ and its various moments. In particular, we have
\[ \sup_{n \in sA} r_{s}(A;n) \leq |A|^{s-1} \ \text{and} \ \sum_{n \in sA} r_{s}(A;n) = |A|^s, \]
whence, 
\[ E_{s,k}(A) \leq (\sup_{n \in sA} r_{s}(A;n) )^{k-1} \sum_{n \in sA} r_{s}(A;n) \leq  |A|^{sk - k + 1}. \]
There are some stronger inequalities that one can obtain between these quantities, and we record some of these as presented in \cite[Lemmata $3.1$ and $3.2$]{Mu2021d}.

\begin{lemma} \label{awk}
Let $s, l, k$ be natural numbers such that $l <s$ and let $A \subset (0, \infty)$ be a finite set. Then
\[ E_{s,2}(A) \leq |A|^{2s - 2l} E_{l,2}(A), \ \text{and} \ M_{s,2}(A) \leq |A|^{2s - 2l} M_{2,2}(A) . \]
Similarly, for all finite sets $A_1, \dots, A_{2s} \subset (0, \infty)$, we have 
\[ \sum_{a_1 \in A_1, \dots, a_{2s} \in A_{2s}} \mathds{1}_{a_1 + \dots + a_s = a_{s+1} + \dots + a_{2s}} \leq E_{s,2}(A_1)^{1/2s} \dots E_{s,2}(A_{2s})^{1/2s}. \]
Finally, when $s$ is even, we have
\[     \sup_{n \in sA} r_{s}(A; n) \leq E_{s/2,2}(A). \]
\end{lemma}

As previously mentioned, our proof of Theorem $\ref{th2}$ will employ various tools from arithmetic combinatorics, the foremost being the following inequality proven by Solymosi \cite{So2009}. 

\begin{lemma} \label{so1}
Let $A \subset (0 , \infty)$ be a finite set such that $|A| \geq 2$. Then 
\[ M_{2,2}(A) \ll |A+A|^2 \log |A|. \]
\end{lemma}

Our next tool of choice will be the Balog--Szemer\'edi--Gowers theorem, as presented in \cite{Sch2015}.

\begin{lemma} \label{bsg5}
Let $A$ be a finite set of real numbers and let $K \geq 1$ be a real number. If $E_{2,2}(A) \geq |A|^{3}/K$, then there exists $A' \subseteq A$ such that $|A'| \gg |A| / K$ and
\[ |A' - A'| \ll K^4 |A'| . \]
\end{lemma}

We will also use the Pl{\"u}nnecke--Ruzsa theorem \cite{Pe2012} to convert the above conclusion concerning difference sets to estimates on sumsets, and so, we record this below.

\begin{lemma} \label{prineq}
Let $A, B$ be finite subsets of some additive abelian group $G$. If $|A+B| \leq K|A|$, then for all non-negative integers $m,n$, we have
$$ |mB - nB| \leq K^{m+n}|A|.$$
\end{lemma}

We will also be utilising incidence geometric techniques in the proof of Theorem $\ref{th1}$ and in order to present these, we introduce some further notation. Thus, given $\vec{u} \in \mathbb{R}^3$, we define the M\"{o}bius transformation $M_{\vec{u}}$ to be
\[ M_{\vec{u}}(x) = \frac{u_1 x + u_2}{x + u_3}. \]
A lot of recent works in incidence theory have focused on studying incidences between a set of M\"{o}bius transformations of the above form and sets of points in $\mathbb{R}^2$. In particular, given a finite set $X \subset \mathbb{R}$ and a finite set $H \subset \mathbb{R}^3$ satisfying $u_2 \neq u_1 u_3$, for each $\vec{u} \in H$, we define
\[ I(X\times X,H) = \sum_{\vec{u} \in H} \sum_{(x_1, x_2) \in X^2} \mathds{1}_{x_2 = M_{\vec{u}}(x_1)}, \] 
whereupon, one may infer from the discussion surrounding \cite[inequality $(8)$]{SS2016} that 
\[ I(X\times X,H) \ll |X|^{4/3} |H|^{2/3} + |X|^{12/11} |H|^{9/11} \log |X| + |X|^2 + |H|. \]
Combining this with \cite[Lemma 3.3]{Mu2021} enables us to present a weighted version of the above result. 
\par

\begin{lemma} \label{wtin}
Let $X \subset \mathbb{R}$ be a finite, non-empty set, and let $H \subset \mathbb{R}^3$ be a finite set such that $u_2 \neq u_1 u_3$, for each $\vec{u} \in H$, and let $w: H \to \mathbb{N}$ be a function. Then
\begin{align*}
\sum_{x_1, x_2 \in X} \sum_{\vec{u} \in H} \mathds{1}_{x_2 = M_{\vec{u}}(x_1)} w(\vec{u}) 
 \ll & \ |X|^{4/3} \big(\sum_{\vec{u} \in H} w(\vec{u})^2 \big)^{1/3} \big(\sum_{\vec{u} \in H} w(\vec{u}) \big)^{1/3} \ + \ \sup_{\vec{u} \in H} w(\vec{u}) |X|^2 \\
& +  |X|^{12/11} \big(\sum_{\vec{u} \in H} w(\vec{u})^2\big)^{2/11} \big(\sum_{\vec{u} \in H} w(\vec{u}) \big)^{7/11} \log |X| \ +  \sum_{\vec{u} \in H} w(\vec{u}).
\end{align*}
\end{lemma}

Finally, we will use a result from extremal graph theory to prove Proposition~\ref{prop: construction}. To state the result, we will first introduce some standard graph theoretic definitions. 
Thus, given a graph $G$, we will use $V(G)$ and $E(G)$ to denote the vertex set of $G$ and the set of edges of $G$ respectively. 
 Given a bipartite graph $H$ and integers $m,n$, the asymmetric bipartite Tur\'an number $\mathrm{ex}(m,n,H)$ of $H$ denotes the maximum number of edges in an $m$ by $n$ bipartite graph that does not contain $H$ as a subgraph. For our purposes, we will set $H = C_{2h}$ for some $h \in \mathbb{N}$, where $C_{2h}$ denotes a $2h$-cycle, that is, $V(C_{2h}) = \{v_1, \dots, v_{2h}\}$ and $E(C_{2h}) = \{ (v_1, v_2), (v_2, v_3) , \dots, (v_{2h}, v_1) \}$. We now record a result of Naor and Verstra\"{e}te \cite{NV05} on bounds for $\mathrm{ex}(m,n,H)$.

\begin{lemma}\label{lem: cycle free}
For $m\leq n$ and $h\geq 2$, we have that
\[
\mathrm{ex}(m,n,C_{2h})\leq 
\begin{cases}
(2h-3)((mn)^{\frac{h+1}{2h}}+m+n)\quad&\text{ if } h \text{ is odd};\\
(2h-3)(m^{\frac{h+2}{2h}}n^{\frac12}+m+n)\quad&\text{ if } h \text{ is even}.
\end{cases}
\]
\end{lemma}

\section{Upper bound constructions}
In this section we will prove Proposition~\ref{prop: construction} via two different constructions, the first of these we present below.

\begin{proof}[Proof of Proposition~\ref{prop: construction}]
We first consider the case when $h$ is even. Let $P$ be a set consisting of the first  $N^{\frac{h}{2h+2}}$ primes, and let $Q$ be a set consisting of the next  $N^{\frac{h+2}{2h+2}}$ primes, and so, $P\cap Q=\emptyset$.
Set
\begin{equation} \label{sur67}
    A:=\{pq\mid p\in P, q\in Q\}. 
\end{equation}
Then $|A|\gg N$, and by way of the Prime number theorem, we have $a\ll N (\log N)^2$ for each $a \in A$.  We first estimate the size of the largest $B_h^+[1]$ subset $B$ of $A$, whereupon, it suffices to note that
\[ |B|^h \ll_h |hB| \leq |hA| \ll_h N (\log N)^2 \]
to prove the required bound.

Next, suppose that $C\subseteq A$ is a $B_h^\times[1]$ set. We construct a bipartite graph $G$ with $V(G)=P\cup Q$, such that given $p \in P$ and $q \in Q$, we have $(p,q) \in E(G)$ if $pq \in C$. Note that for every $h$ distinct elements $p_1, \dots, p_h\in P$ and every $h$ distinct elements $q_1,\dots,q_h\in Q$, the following set
\[
\{p_1q_1, p_2q_2, \dots, p_hq_h, p_1q_2, p_2q_3, \dots, p_{h-1}q_h, p_hq_1\}
\]
is not contained in $C$, since the product of the first $h$ elements is equal to the product of the last $h$ elements in the above set. This implies that our graph $G$ is $C_{2h}$-free, and so, we may apply Lemma~\ref{lem: cycle free} to deduce that the number of edges $|E(G)|$ of our graph satisfies
\begin{align*}
|E(G)|\leq \mathrm{ex}(|P|,|Q|,C_{2h})\ll_h N^{\frac{h+2}{2h+2}}= N^{\frac12+\frac{1}{2h+2}}.
\end{align*}
The desired bound then follows from noting that $|C| \leq |E(G)|$, which holds true because each element in $C$ has a unique representation as a product of two primes.
\par

Finally, the case when $h$ is odd follows the fact that a $B_{s}[1]$ set is also a $B_{s-1}[1]$ set for every $s \geq 3$.
\end{proof}

More generally, in the case when $A = \{ pq \ | \ p \in P, q \in Q\}$ for some suitable sets $P$ and $Q$ of primes, the problem of estimating the size of the largest $B_h^{+}[1]$ and $B_h^{\times}[1]$ sets in $A$ is equivalent to finding the largest bipartite graph on vertex sets $P$ and $Q$ without cycles of length $\leq 2h$, which in turn is closely connected to finding sharp estimates for $\mathrm{ex}(|P|,|Q|, C_{2h})$. Furthermore, for large values of $h$, finding optimal lower bounds for the latter function for different regimes of $|P|, |Q|$ is a major open problem; we refer the reader to \cite[Section 4]{FS13} and the references therein for a survey of this topic. 

Instead of using prime numbers, one can obtain similar results using powers of $2$, and we briefly sketch this as follows. 

\begin{proof}[An alternative proof of Proposition $\ref{prop: construction}$]
Let $h,n,M,N$ be even natural numbers such that $N= n^{h}$ and $M= n^{h+2}$. Moreover,  let $P_h = \{1,2, \dots, 2^N\}$ and $Q_{h} = \{2^{N+1}, \dots, 2^{N+M}\}$ be geometric progressions and let $A_h = P_h + Q_h$. Note that $|A_h| \gg_h |P_h||Q_h| \gg_h n^{2h+2}.$ 
\par

Given a $B_h^+[1]$ set $B \subseteq A$, we may construct a bipartite graph $G$ on $P_h \times Q_h$ by letting $(p,q)\in E(G)$ if $p+q \in B$. This implies that $G$ must be $C_{2h}$-free,
whence, we may apply Lemma \ref{lem: cycle free} to deduce that
\[ |B| \ll_h |P_h|^{\frac{h+2}{2h}} |Q_h|^{1/2} + |Q_h| \ll_h n^{h+2} \ll_h |A|^{\frac{h+2}{2h+2}}. \]
\par

Similarly, let $C$ be a $B_{h}^{\times}[1]$ set in $A$. Here, we note that $A \subseteq P_h \cdot Q'_h,$ where $Q'_h = \{1 + 2^j \ | \ 1 \leq j \leq N+M\}$. We may now construct a bipartite graph $G'$ on $P_h \times Q_h'$ by letting $(p',q') \in G$ if $p'q' \in C$. As before, we may observe that the graph $G'$ is $C_{2h}$-free, whenceforth, Lemma \ref{lem: cycle free} delivers the bound
\[ |C|  \ll_h |P_h|^{\frac{h+2}{2h}} |Q'_h|^{1/2} + |Q'_h| \ll_h n^{h+2} \ll_h |A|^{\frac{h+2}{2h+2}}. \]

Finally, the case when $h$ is odd follows trivially from the case when $h$ is even.
\end{proof}

We note that the set $A_2$ was recorded in work of Erd\H{o}s \cite[page $57$]{Er1983}, who used this set to prove a related conjecture on the size of the largest $B_{2}^+[1]$ set contained in a $B_{2}^+[2]$ set, and subsequently, Shkredov \cite{Sh2021} proved that the set $A_2$ also refutes the aforementioned conjecture of Klurman--Pohoata.

We now outline a construction of Balog--Wooley \cite{BW2017}, which was later modified by Roche-Newton to show that there exist sufficiently large subsets $A$ of $\mathbb{N}$ such that the largest $B_{2}^+[1]$ and $B_{2}^\times[1]$ sets in $A$ have size at most $O(|A|^{3/4})$. Thus, letting $h \geq 2$ and 
\[ A_{M,N} = \{ (2i+1)2^j \ | \ 1 \leq i \leq M \ \text{and} \ 1 \leq j \leq N \} , \]
we will show that the largest $B_{h}^+[g]$ and $B^\times_h[g]$ subsets of $A_{N,N}$ have size at most $O_{g,h}(N^{\frac{h+1}{h}})$. A straightforward application of pigeonhole principle allows us to deduce that any subset $B \subseteq A_{N,N}$ satisfying $|B| \geq 2gh!hN^{\frac{h+1}{h}}$ contains at least $gh!hN^{1/h}$ elements of $2^{j+1} \cdot [N] + 2^j$ for some $j \in \mathbb{N}$, and so, $B$ can not be a $B_{h}^+[g]$ set due to the fact that $\eqref{erds2}$ holds true. On the other hand, any $B_h^{\times}[g]$ set $C \subseteq A_{N,N}$ satisfies
\[ |C|^h \leq g h! |A_{N,N}^{(h)}| \leq gh!h N^{h+1}, \]
and so, we are done.
\par

It was noted by Balog--Wooley \cite{BW2017} that sets of the form $A_{M,N}$ restrict how good a power saving one can obtain in results akin to Lemma $\ref{mu1}$. More specifically, they showed that any subset $B \subseteq A_{N^2,N}$ with $|B| \geq N^3/2$ satisfies
\[ E_{s,2}(B) \gg_{s} |B|^{s + (s-1)/3} \ \text{and} \ M_{s,2}(B) \gg_{s} |B|^{s + (s-1)/3}. \]


\section{Solving simultaneous linear equations with repetitive elements}

Our main aim in this section is to estimate the number of solutions to a simultaneous system of equations where there are restrictions on the number of distinct elements in each solution. We begin this endeavour by presenting some further notation, and thus, for any $l, k,s \in \mathbb{N}$ satisfying $1 \leq l \leq ks$ and for any finite, non-empty set $A$ of real numbers, we denote the vector 
\[ (a_{1,1}, \dots, a_{1,s}, a_{2,1}, \dots, a_{2,s}, \dots, a_{k,1}, \dots, a_{k,s}) \in A^{ks} \]
 to be \emph{$(k,l)$-complex} if there are precisely $l$ distinct values in the set $\{ a_{1,1}, \dots, a_{k,s}\}$ and if for any $1 \leq i < j \leq k$, we have that $\{a_{i,1}, \dots, a_{i,s} \} \neq \{a_{j,1} , \dots, a_{j,s}\}$. Moreover, we use $W_{k,l}$ to denote the set of all $(k,l)$-complex vectors in $A^{ks}$. Next, let $\Sigma_{l,s,k}(A)$ count the number of solutions to the system of equations 
\[
a_{1,1}+\cdots+a_{1,s}=a_{2,1}+\cdots+a_{2,s} = \cdots=a_{k,1}+\cdots+a_{k,s}, 
\]
where $(a_{1,1}, \dots, a_{k,s}) \in W_{k,l}$ such that for any $1 \leq i < j \leq k$, we have that $\{a_{i,1}, \dots, a_{i,s} \} \neq \{a_{j,1} , \dots, a_{j,s}\}$. The main task of the section is to estimate $\Sigma_{l,s,k}(A)$ under the assumption that $E_{s,k}(A)$ is bounded. We note that the above system may be rewritten as the following system of $k-1$ simultaneous linear equations 
\begin{equation} \label{alt}
a_{i,1} + \dots + a_{i,s} -  a_{k,1} - \dots - a_{k,s} = 0 \ \ \ (1 \leq i \leq k -1).
\end{equation}
We will often write $E_i = a_{i,1} + \dots + a_{i,s}$ for each $1 \leq i \leq k$.  
\par

The next two lemmata provide estimations of $\Sigma_{l,s,k}(A)$ when either $k=2$ or $s=2$. 
 \begin{lemma}\label{lem: sigma l,s,2}
Let $s,l$ be natural numbers such that $2 \leq l \leq 2s$. Moreover suppose that $A$ is a finite set of real numbers such that
\[ E_{s,2}(A) \ll_{s} |A|^{2s - 2 + 1/s - c}, \]
for some $c>0$. Then we have that
\[ \Sigma_{l,s,2}(A) \ll_{s} |A|^{l - l/s + l/2s^2 - cl/2s}. \]
 \end{lemma}
 \begin{proof}
Writing $f(\alpha) = \sum_{a \in A} e(\alpha a)$ for every $\alpha \in [0,1)$, we may use orthogonality to deduce the following inequality
\begin{align*}
  \Sigma_{l,s,2}(A)
  \ll_{s} \sum_{\substack{0 < |c_1|, \dots, |c_l| \leq 2s, \\ c_1 + \dots + c_l = 0 }}\int_{[0,1)} f(c_1\alpha) \dots f(c_l \alpha) d \alpha.
\end{align*}
Applying H\"{o}lder's inequality and periodicity, we see that
\begin{align*}
  \Sigma_{l,s,2}(A)
  \ll_{s} \prod_{i=1}^{l} \Big(\int_{[0,1)} |f(c_i \alpha)|^{2s} d \alpha \Big)^{1/2s} = \Big(\int_{[0,1)} |f(\alpha)|^{2s} d \alpha \Big)^{l/2s},
\end{align*}
whereupon, we obtain the bound
\begin{equation*}
 \Sigma_{l,s,2}(A) \ll_{s}   E_{s,2}(A)^{l/2s} \ll_{s} |A|^{l - l/s + l/2s^2 - cl/2s},
 \end{equation*} 
 which proves the lemma. 
 \end{proof}

  \begin{lemma}\label{lem: sigma l,2,k}
Let $k$ be a natural number and let $A$ be a finite set of real numbers such that
\[ E_{2,k}(A) \ll_{k} |A|^{k + 1/2 - c}, \]
for some $c>0$. Then we have that
\[ \Sigma_{2k-1,2,k}(A) \ll_{k} |A|^{k-1/2+1/2k-c(1-1/k)}. \]
 \end{lemma}
 \begin{proof}
Without loss of generality we can assume that $a_{i,1} = a_{i,2}$ for some $1 \leq i \leq k$. We now apply H\"{o}lder's inequality to get
\begin{align*}
 \Sigma_{2k-1,2,k}(A) & \ll_k \sum_{x \in 2A} \big(\sum_{a_1, a_2 \in A} \mathds{1}_{x = a_1 + a_2}\big)^{k-1} \mathds{1}_{2\cdot A}(x) \\
&  \leq \big( \sum_{x} \big(\sum_{a_1, a_2 \in A} \mathds{1}_{x = a_1 + a_2}\big)^k\big)^{1 - 1/k} |A|^{1/k}  \\
& = E_{2,k}(A)^{1 - 1/k} |A|^{1/k} ,
\end{align*}
which, when combined with the hypothesis recorded above, delivers the required bound
\[ \Sigma_{2k-1,2,k}(A) \ll_{k} |A|^{k - 1/2 + 1/2k - c (1 - 1/k)  }. \qedhere \]
 \end{proof}

In the remaining parts of this section, we will focus on estimating $\Sigma_{l,k,s}(A)$ for a much more general range of $k,s$, and we begin this endeavour by presenting the following straightforward upper bound on the number of solutions to a system of linear equations of a given rank with all the variables lying in some prescribed set.

\begin{lemma} \label{lin}
Let $m,n,r$ be natural numbers, let $M$ be a $m \times n$ matrix with real coefficients, let $\vec{u}=(u_1,\dots,u_m)$ be some vector in $\mathbb{R}^m$ and let $A$ be a finite, non-empty set of real numbers. Suppose that the matrix $M$ has $r$ linearly independent rows. Then the number of solutions to 
\[ M \vec{a}^T = \vec{u}^T, \]
with $\vec{a} = (a_1, \dots, a_n) \in A^n$ is at most $O(|A|^{n-r})$. 
\end{lemma}


\begin{proof}
We apply Gaussian elimination on $M$ and obtain its row echelon form $M'=PM$ where $P$ is a $m\times m$ matrix. Note that the first $r$ rows in $M'$ are linearly independent and upper triangular, and the other $m-r$ rows are $\vec{0}$. Let $\vec{v}_i$ be the $i$-th row vector of $M'$. Without loss of generality we may assume that the $i$-th entry of $\vec{v}_i$ is non-zero for every $1\leq i\leq r$. Let $\vec{u}'=(u_1',\dots,u_m')$ be $P\vec{u}^T$. 
Since the solutions to $M \vec{a}^T = \vec{u}^T$ are the solutions to $M' \vec{a}^T =\vec{u}'^T$, by fixing $(a_{r+1},\dots,a_n)\in A^{n-r}$, we have
\[
(M')_{r\times r} \vec{a}_r^T={\vec{u}'}_r^T,
\]
where $(M')_{r\times r}$ contains the first $r$ rows and $r$ columns of $M'$, $\vec{a}_r=(a_1,\dots,a_r)$ and $\vec{u}'_r=(u'_1,\dots,u'_r)$. By the assumption, $(M')_{r\times r}$ has full rank, and hence  $\vec{a}_r=(a_1,\dots,a_r)\in \mathbb{R}^r$ can be uniquely determined. Finally, since there are at most $|A|^{n-r}$ ways to choose $(a_{r+1},\dots,a_n) \in A^{n-r}$, the desired conclusion follows. 
\end{proof}

We finish this section by presenting the following lemma that enables us to find appropriate bounds for $\Sigma_{l,s,k}(A)$ when $s,k$ are natural numbers with $k \geq 3s$.

\begin{lemma} \label{lim2}
Let $s,k,l$ be natural numbers such that $k \geq 3s$ and $2 \leq l \leq sk$. Moreover suppose that $A$ is a finite set of real numbers such that
\[ E_{s,k}(A) \ll_{s,k} |A|^{sk - k + 1/s - c}, \]
for some $c>0$. Then we have that
\[ \Sigma_{l,s,k}(A) \ll_{s,k} |A|^{l - l/s + 1/s - c'}, \]
for some $c'\geq \min\{(k-2s)c/k,1/s\}$. 
\end{lemma}

\begin{proof}[Proof of Lemma $\ref{lim2}$]
For ease of exposition, we will write $\Sigma_{l} = \Sigma_{l,s,k}(A)$, suppressing the dependence on $s,k$ and $A$. Let $M$ be the coefficient matrix of the system of linear equations described in $\eqref{alt}$, and in particular, $M$ will be some $(k-1) \times l$ matrix with entries from $[-2s, 2s] \cap \mathbb{Z}$. By incurring a factor of $O_{s,k}(1)$ in our upper bounds, which subsequently gets absorbed in the implicit constant of the Vinogradov notation, we may fix all the entries in $M$. 
\par

We divide our proof into two cases depending on the rank of the matrix $M$, and so, we first consider the case when $\textrm{rank}(M) = k-1$. Furthermore, in this setting, it suffices to analyse the situation when $l \in (s(k-1), sk]$, since otherwise, we may use Lemma $\ref{lin}$ to deduce that
\[
\Sigma_l \ll_{s,k} |A|^{l-k+1}=|A|^{l-\frac{l-1}{s}-\frac{(sk-l)-(s-1)}{s}},
\]
Thus, we assume that $s(k-1) < l \leq sk$. In this case, there are $sk-l$ repetitive variables, all of which lie in $d$ different $s$-tuples, for some $d\leq 2(sk-l)$. By losing a factor of $O_{s,k}(1)$, we may assume that the $d$ $s$-tuples which contain the repetitive elements are precisely $(a_{1,1},\dots,a_{1,s})$, $\dots$,  $(a_{d,1},\dots,a_{d,s})$. Thus, we have that
\begin{align*}
    \Sigma_{l} 
    & \ll_{s,k} \sum_{n} \sum_{a_{d+1,1}, \dots, a_{k,s} \in A} \mathds{1}_{E_{d+1} = \dots = E_{k} = n} \sum_{\vec{a} \in W_{d,l}} \mathds{1}_{E_{1} = \dots = E_{d} =n}  \\
    & = \sum_{n} r_{s}(A;n)^{k-d} \sum_{\vec{a} \in W_{d,l}} \mathds{1}_{E_{1} = \dots = E_{d} =n} ,
\end{align*} 
where we use $\vec{a} \in W_{d,l}$ to denote the element $(a_{1,1}, \dots, a_{d,s}) \in W_{d,l}$.
Applying H\"older's inequality, we get that
\begin{equation} \label{kmov}
\Sigma_l \ll_{s,k} E_{s,k}(A)^{\frac{k-d}{k}}\Big(\sum_{n}\Big(\sum_{\vec{a} \in W_{d,l}} \mathds{1}_{E_{1} = \dots = E_{d} =n} \Big)^{\frac{k}{d}}\Big)^{\frac{d}{k}}.
\end{equation}
\par

Using Lemma $\ref{lin}$ along with the fact that $k \geq 2s > 2(sk-l) \geq d$, we may conclude that
\begin{equation} \label{j1i1}
    \sum_{\vec{a} \in W_{d,l}} \mathds{1}_{E_{1} = \dots = E_{d}} \ll_{s,k} |A|^{sd - (sk-l) - (d-1)},
\end{equation} 
as well as that
\begin{equation} \label{j2i2}
 \sum_{\vec{a} \in W_{d,l}} \mathds{1}_{E_{1} = \dots = E_{d} =n} \ll_{s,k} |A|^{sd - (sk - l) -  d  } 
\end{equation} 
holds for every $n \in \mathbb{R}$. More specifically, in order to prove $\eqref{j1i1}$, note that the system $E_1 = \dots = E_{d}$ can be rewritten in the form $\eqref{alt}$, wherein, the associated matrix has rank $d-1$. This follows from the fact that $\mathrm{rank}(M)=k-1$. Moreover, since there are exactly $sd - (sk-l)$ distinct elements in each solution, we can now use Lemma $\ref{lin}$ to deliver the claimed inequality. The deduction of the second inequality from Lemma $\ref{lin}$ requires some further maneuvers, which we briefly record here. The reader will note that it suffices to show that the row vectors $\vec{C}_1, \dots, \vec{C}_d$ of the matrix affiliated with the system $E_1 = \dots = E_d = n$ are linearly independent. We prove this via contradiction, and so, without loss of generality, we may suppose that $c_1, \dots, c_{d-1}$ are real numbers satisfying 
\[ \vec{C}_{d} = \sum_{i = 1}^{d-1} c_i \vec{C}_{i}. \]
Multiplying the above equation by $\vec{v}^T$, where $\vec{v} = (1,\dots, 1) \in \mathbb{R}^l$, and employing the fact that $\vec{C}_i \vec{v}^T = s$ for each $1 \leq i \leq s$, we deduce that $\sum_{i=1}^{d-1} c_i = 1$. But this allows us to write
\[ \vec{R}_{d} = \sum_{i=1}^{d-1} c_i \vec{R}_i, \]
where $\vec{R}_1, \dots, \vec{R}_{k-1}$ are the row vectors of the matrix $M$, thus contradicting the fact that $\mathrm{rank}(M)=k-1$.
\par

Combining $\eqref{j1i1}$ and $\eqref{j2i2}$ with $\eqref{kmov}$, we get that
\begin{align*}
\Sigma_l 
& \ll_{s,k} |A|^{(s-1)(k-d)+\frac{k-d}{sk}-\frac{(k-d)c}{k}}|A|^{(sd-sk+l-d)\frac{k-d}{k}} \Big(\sum_{\vec{a} \in W_{d,l}} \mathds{1}_{E_{1} = \dots = E_{d}} \Big)^{\frac{d}{k}}\\
& \ll_{s,k} |A|^{(s-1)(k-d)+\frac{k-d}{sk}-\frac{(k-d)c}{k}}|A|^{(sd-sk+l-d)\frac{k-d}{k}}|A|^{(sd-sk+l-d+1)\frac{d}{k}}\\
& \leq |A|^{l-\frac{l-1}{s}-\frac{(sk-l)(k-2s+2)}{sk}-\frac{(k-d)c}{k}}\leq |A|^{l-\frac{l-1}{s}-c'},
\end{align*}
with $c'\geq c(k-2s)/k$, whereupon, we are done when $\mathrm{rank}(M)=k-1$.
\par

Thus, we proceed with our second case, that is, when $\mathrm{rank}(M)=r<k-1$. This already implies that $l \leq s(r+1)$, and in fact, we will show that the stronger bound $l \leq sr$ must hold. This, in turn, combines with Lemma $\ref{lin}$ to deliver the estimate
\[
\Sigma_l\ll |A|^{l-r}=|A|^{l-\frac{l-1}{s}-\frac{sr-l+1}{s}}=|A|^{l-\frac{l-1}{s}-c'}
\]
where $c'\geq 1/s$. We now turn to proving that our claim holds, that is, $l \leq sr$. Without loss of generality, we may assume that the first $r$ rows in $M$ are linearly independent. Since $r  \leq k-1$, there exist $\alpha_1, \dots, \alpha_r \in \mathbb{R}$ such that
\[
\vec{R}_{k-1}=\sum_{i=1}^r \alpha_i \vec{R}_i. 
\]
 Let $I,J\subseteq[r]$ be sets such that $\alpha_i>0$ for $i\in I$ and $\alpha_i<0$ for $i\in J$, and let $K = [r]\setminus (I \cup J)$. 
\par

As all the $s$-tuples that we are analysing correspond to essentially distinct representations, we have that $|I|,|J|\geq 1$, whence
\[ |K| + |I| \leq r-1 . \]
 Writing $\beta_j =-\alpha_j$ for each $j \in J$, we get that
\[
\vec{R}_{k-1}=\sum_{i \in I}\alpha_i \vec{R}_i-\sum_{j \in J}\beta_j \vec{R}_j.
\]
Thus, setting $F_i = \vec{R}_{i} \vec{x}^T$ where $\vec{x} = (x_1, \dots, x_l)$ is a vector with formal variables $x_1, \dots, x_l$ as entries, we may deduce from the preceding expression that
\[ F_{k-1} - F_k = \sum_{i \in I} \alpha_i (F_i - F_k) - \sum_{j \in J} \beta_j (F_j - F_k), \]
and so,
\[  \sum_{j \in J} \beta_j F_j = \sum_{i \in I} \alpha_i F_i - F_{k-1} + (\sum_{j \in J} \beta_j - \sum_{i \in I} \alpha_i + 1) F_k .\]
Since $\alpha_i, \beta_j > 0$ for each $i \in I$ and $j \in J$, we must have that any variable appearing in $F_j$, for every $j \in J$, either occurs in $F_i$ for some $i \in I$ or it occurs in $F_{k}$. Thus, we deduce that all the distinct variables arise either from $F_i$, for some $i \in I \cup K$, or from $F_{k}$. Finally, as $l$ is bounded above by the number of distinct variables in $F_1, \dots, F_r$ , we infer that
\[ l \leq s(|I| + |K|) + s \leq s(r-1) + s = rs, \]
and so, our claim holds true. This finishes the proof of Lemma $\ref{lim2}$. 
\end{proof}


\section{Random sampling and deletion}

We will use this section to record various lemmata that connect bounds on additive and multiplicative energies to the existence of large $B_{s}[g]$ subsets.

\begin{lemma} \label{gens}
Let $A \subset \mathbb{N}$ be a finite set, let $s \geq 2$ be a natural number and let $c >0$ be a real number such that
\[ E_{s,2}(A) \leq |A|^{2s - 2 + 1/s - c}. \]
Then there exists $B \subseteq A$ such that $B$ is a $B_{s}^+[1]$ set satisfying 
\[ |B| \gg_{s} |A|^{1/s + \delta} \ \text{for} \ \delta = c/(2s). \]
\end{lemma}

\begin{proof}
We begin our proof by applying Lemma \ref{lem: sigma l,s,2} to deduce that
\begin{equation} \label{hld} 
 \Sigma_{l,s,2}(A) \ll_{s}   E_{s,2}(A)^{l/2s} \ll_{s} |A|^{l - l/s + l/2s^2 - cl/2s},
 \end{equation}
 for each $2 \leq l \leq 2s$.
We will now pick elements from $A$ with probability $p$ uniformly at random, where $p = |A|^{1/s - 1 + \delta}$, and we write this subset to be $A'$. Note that
\[ \mathbb{E} |A'| = p |A| = |A|^{1/s + \delta}, \]
as well as that 
\[ \mathbb{E} |A'| - 2\mathbb{E} \sum_{l=2}^{2s} \Sigma_{l,s,2}(A') =p|A| - 2\sum_{l=2}^{2s} p^l \Sigma_{l,s,2}(A) = |A|^{1/s + \delta} - O_{s}(\sup_{2 \leq l \leq 2s} |A|^{l\delta + l/2s^2 - cl/2s}),\]
where the last inequality follows from $\eqref{hld}$.
Our choice of $\delta$ now implies that
\[ \mathbb{E} ( |A'| - 2 \sum_{l=2}^{2s} \Sigma_{l,s,2}(A') ) \geq |A|^{1/s + \delta}/2, \]
whenever $|A|$ is sufficiently large in terms of $s$. Thus, there exists some $A' \subseteq A$ such that
\[ |A'| \geq |A|^{1/s + \delta}/2 \ \text{and} \ \sum_{l=2}^{2s} \Sigma_{l,s,2}(A') \leq |A'|/2 . \]
\par

For each $2 \leq l \leq 2s$ and for each solution $(a_1, \dots, a_{2s})$ counted in $\Sigma_{l,s,2}(A')$, we remove the element $a_1$ from $A'$, and we denote $B$ to be the remaining set. By definition, the set $B$ must be a $B_{s}^+[1]$ set. Moreover, we have that
\[ |B| \geq |A'| - \sum_{2 \leq l \leq 2s} \Sigma_{l,s,2}(A') \geq |A'|/2 \geq |A|^{1/s + \delta}/4, \]
and so, we are done.
\end{proof}

It can be shown that Lemma $\ref{gens}$ also holds for multiplicative energies and multiplicative $B_{s}[1]$ sets, but we have to apply some slight modifications to various parts of the proof. 


\begin{lemma} \label{mu2}
Let $s$ be a natural number, let $c>0$ and let $A \subset (0, \infty)$ be a finite set such that 
\[ M_{s,2}(A) \leq |A|^{2s - 2 + 1/s - c}, \]
 then there exists $A' \subseteq A$ such that $A$ is a $B_{s}^\times[1]$ set satisfying
\[ |B| \gg_s |A|^{1/s + \delta} \ \text{for} \ \delta = c/(2s). \]
\end{lemma}

\begin{proof}
For every $2 \leq l \leq 2s$, let $\Pi_{l,s,2}(A)$ be the number of all $2s$-tuples $(a_1, \dots, a_{2s}) \in A^{2s}$ satisfying $a_1  \dots a_s = a_{s+1}  \dots  a_{2s}$ such that there are precisely $l$ distinct elements amongst $a_1, \dots, a_{2s}$. Our main aim is to show that for each $2 \leq l \leq 2s$, we have
\begin{equation} \label{ann}
 \Pi_{l,s,2}(A) \ll_{s} |A|^{l - l/s + l/2s^2 - cl/2s}, 
 \end{equation}
since we can then follow the proof of Lemma $\ref{hld}$ mutatis mutandis to deduce our desired claim. 
\par

We begin our proof of $\eqref{ann}$ by noting that
\[ \Pi_{l,s,2}(A) \ll_{s} \sum_{\substack{ 0 < |c_1|, \dots, |c_{l}| \leq 2s, \\ c_1 + \dots + c_l = 0  }} \sum_{a_1, \dots, a_l \in A} \mathds{1}_{c_1 \log a_1 + \dots + c_l \log a_l = 0 } . \]
Writing $X = \{ \log a \ | \ a \in A\}$, we let $A_i = c_i \cdot X$ for every $1 \leq i \leq \min\{l, s\}$ and $A_i = - c_i \cdot X$ for every $s+1 \leq i \leq l$ and $A_i = \{0\}$ for every $l+1 \leq i \leq 2s$. Thus, the previous inequality may be rewritten as
\[ \Pi_{l,s,2}(A) \ll_{s} \sum_{\substack{ 0 < |c_1|, \dots, |c_{l}| \leq 2s, \\ c_1 + \dots + c_l = 0  }} \sum_{a_1 \in A_1, \dots, a_{2s} \in A_{2s}} \mathds{1}_{a_1 + \dots + a_s = a_{s+1} + \dots + a_{2s}}, \]
whence, we may apply Lemma $\ref{awk}$ to obtain the bound
\[ \Pi_{l,s,2}(A) \ll_{s} \sum_{\substack{ 0 < |c_1|, \dots, |c_{l}| \leq 2s, \\ c_1 + \dots + c_l = 0  }} E_{s,2}(A_1)^{1/2s} \dots E_{s,2}(A_l)^{1/2s} . \]
Finally, since the equation $x_1 + \dots + x_s = x_{s+1} + \dots + x_{2s}$ is dilation invariant, and recall that $X=\log A$, we see that $E_{s,2}(A_i) = E_{s,2}(X) = M_{s,2}(A)$, and subsequently, we get the bound
\[ \Pi_{l,s,2}(A) \ll_{s} M_{s,2}(A)^{l/2s} \sum_{\substack{ 0 < |c_1|, \dots, |c_{l}| \leq 2s, \\ c_1 + \dots + c_l = 0  }} 1 \ll_{s} |A|^{l - l/s + l/2s^2 - cl/2s}  . \qedhere \]
\end{proof}

We now prove similar results when we have good upper bounds for either $E_{2,k}(A)$ or $M_{2,k}(A)$.

\begin{lemma} \label{sidr}
Let $k \geq 2$ be a natural number, let $c, \delta>0$ be real numbers such that $\delta = c/2k$ and let $A \subset (0, \infty)$ be a finite set. If 
\[ E_{2,k}(A) \ll_{k} |A|^{k+ 1/2 - c},\]
 then there exists a $B_{2}^+[k-1]$ set $B \subseteq A$ such that $|B| \gg_k |A|^{1/2 + \delta}$. Similarly, if 
 \[ M_{2,k}(A) \ll_{k} |A|^{k+ 1/2 - c}, \]
  then there exists a $B_{2}^\times[k-1]$ set $B \subseteq A$ such that $|B| \gg_k |A|^{1/2 + \delta}$.
\end{lemma}

\begin{proof}
Recall that for every $2 \leq l \leq 2k$, the quantity $\Sigma_{l,2,k}(A)$ was defined to be the number of $2k$-tuples $(a_1, \dots, a_{2k}) \in A^{2k}$ satisfying $a_1 + a_2 = \dots = a_{2k-1} + a_{2k}$ such that there are precisely $l$ distinct elements amongst $a_1, \dots, a_{2k}$. We first claim that it suffices to consider the case when $l\geq 2k-1$. In order to see this, note that if at least three of $a_{1,1}, \dots, a_{k,2}$ equal each other, then without loss of generality, we may assume that there exist $1 \leq i < j \leq k$ such that $a_{i,2} = a_{j,2}$. But this would then imply that $a_{i,1} = a_{j,1}$, since $a_{i,1} + a_{i,2} = a_{j,1} + a_{j,2}$, whence, $\{ a_{i,1}, a_{i,2}\} = \{ a_{j,1}, a_{j,2}\}$, which contradicts our setting wherein we are only interested in distinct representations of some real number $n$ as $n = a+ b$, with $a, b \in A$.

By Lemma~\ref{lem: sigma l,2,k}, we have
\[ \Sigma_{2k-1,2,k}(A)\ll_{k} |A|^{k - 1/2 + 1/2k - c (1 - 1/k)  } , \]
and furthermore, we have the trivial bound
 \[ \Sigma_{2k,2,k}(A) \ll_k E_{2,k}(A) \ll_{k} |A|^{k + 1/2 - c} . \]
As before, we now use a random sampling argument, to pick elements $a \in A$ with probability $p = |A|^{-1/2 + \delta}$ uniformly at random, and we denote this set to be $A'$. Thus, we have that
\[ \mathbb{E}|A'| = p|A| = |A|^{1/2 + \delta}. \]
Furthermore, we see that
\begin{align*}
  \mathbb{E}|A'| -  2 \mathbb{E} \sum_{l = 2k-1}^{2k} \Sigma_{l,2,k}(A')
 &  = p |A| - 2p^{2k} \Sigma_{2k,2,k}(A) - 2p^{2k-1}  \Sigma_{2k-1,2,k}(A) \\
  & = |A|^{1/2 + \delta} - O_{k}(|A|^{1/2  - c + 2k \delta} + |A|^{1/2k - c(1-1/k) + (2k-1) \delta} ).
  \end{align*}
Since $k \geq 2$ and $\delta = c/2k$, both the error terms above can be verified to be much smaller than the main term, and consequently, we get that 
  \[   \mathbb{E} (|A'| -  2 \sum_{l = 2k-1}^{2k} \Sigma_{l,2,k}(A')) \geq |A|^{1/2 + \delta}/2 \]
whenever $|A|$ is sufficiently large in terms of $k$. This implies that there exists some $A' \subseteq A$ such that
  \[ |A'| \geq |A|^{1/2 + \delta}/2, \ \text{as well as that} \  \Sigma_{2k-1,2,k}(A') + \Sigma_{2k,2,k}(A') \leq |A'|/2. \]
For each solution $(a_1, \dots, a_{2k})$ counted by either $\Sigma_{2k-1,2,k}(A')$ or $\Sigma_{2k,2,k}(A')$, we remove the element $a_1$ from $A'$, and we write the remaining set to be $B$. By definition, $B$ is a $B_{2}^+[k-1]$ set satisfying $|B| \geq |A'|/2 \gg_{k} |A|^{1/2 + \delta}$, and so, we have proven the first conclusion recorded in Lemma $\ref{sidr}$. The multiplicative analogue can be shown to hold similarly by applying the first part of Lemma $\ref{sidr}$ for the sets $X_1 = \{ \log a \ | \ a \in A \cap (0, 1) \}$ and $X_2 = \{ \log a \ | \ a \in A \cap (1, \infty)\}$.
\end{proof}

Next, we will also show that similar arguments imply that whenever $E_{s,k}(A)$ is bounded appropriately for some $k \geq 3s$, then there exists a large $B_s^+[k-1]$ set in $A$.

\begin{lemma} \label{hsh}
Let $A$ be a finite set of real numbers and let $s,k \geq 2$ be natural numbers, $k\geq 3s$, and let $c >0$. If 
\[ E_{s,k}(A) \ll |A|^{sk - k + 1/s - c}, \]
then there exists a $B_{s}^+[k-1]$ set $B \subseteq A$ such that $|B| \gg_k |A|^{1/s + \delta}$, for $\delta = c'/sk$ with $c'= \min\{(k-2s)c/k,1/s\}$.
\end{lemma}
\begin{proof}
We begin our proof by applying Lemma \ref{lim2} to deduce that
\begin{equation*} 
 \Sigma_{l,s,k}(A) \ll_{s,k}|A|^{l-l/s+1/s-c'}.
 \end{equation*}
We now pick elements from $A$ with probability $p$ uniformly at random, where $p = |A|^{1/s - 1 + \delta}$, and we write this subset to be $A'$. As
$ \mathbb{E} |A'| = p |A| = |A|^{1/s + \delta}, $
we have that 
\[ \mathbb{E} |A'| - 2\mathbb{E} \sum_{l=2}^{ks} \Sigma_{l,s,k}(A') =p|A| - 2\sum_{l=2}^{2s} p^l \Sigma_{l,s,k}(A) = |A|^{1/s + \delta} - O_{s}(|A|^{sk\delta + 1/s - c'}).\]
Our choice of $\delta$ now implies that
\[ \mathbb{E} ( |A'| - 2 \sum_{l=2}^{ks} \Sigma_{l,s,k}(A') ) \geq |A|^{1/s + \delta}/2, \]
whenever $|A|$ is sufficiently large in terms of $s$. Thus, there exists some $A' \subseteq A$ such that
\[ |A'| \geq |A|^{1/s + \delta}/2 \ \text{and} \ \sum_{l=2}^{ks} \Sigma_{l,s,k}(A') \leq |A'|/2 . \]
\par

For each $2 \leq l \leq ks$ and for each solution $(a_{1,1}, \dots, a_{k,s})$ counted by $\Sigma_{l,s,k}(A')$, we remove the element $a_{1,1}$ from $A'$, and we denote $B$ to be the remaining set. By definition, the set $B$ must be a $B_{s}^+[k-1]$ set. Moreover, we have that
\[ |B| \geq |A'| - \sum_{2 \leq l \leq ks} \Sigma_{l,s,k}(A') \geq |A'|/2 \geq |A|^{1/s + \delta}/4. \qedhere \]
\end{proof}


\section{Hyperbolic incidences}


Given finite, non-empty sets $X,Y \subset \mathbb{R}$, we are interested in estimating the number of solutions $H(X,Y)$ to the equation
\[ (x_1 - y_1) (x_2 - y_2)  = 1, \]
with $x_1, x_2 \in X$ and $y_1, y_2 \in Y$. By dilating the sets $X,Y$ appropriately, we may use this to study solutions to equations of the form $(x_1 - y_1)(x_2 - y_2) = \lambda$, for arbitrary $\lambda \neq 0$.
Our main goal in this section is to prove the following upper bound for $H(X,Y)$.

%

\begin{theorem} \label{hyp}
Let $X,Y \subset \mathbb{R}$ be finite sets such that $|Y|^2 \leq |X| \leq |Y|^3$. Then we have
\[ H(X,Y) \ll |X|^{1 + 1/6} |Y|^{2 - 1/2}  . \]
\end{theorem}

\begin{proof}
Let $X,Y$ be finite subsets of $\mathbb{R}$ such that $|Y|^2 \leq |X| \leq |Y|^3$. Note that
\begin{align*}
    H(X,Y) = \sum_{x_1, x_2, y_1, y_2} \mathds{1}_{x_2 = y_2 + (x_1 - y_1)^{-1}} = \sum_{u} (\sum_{x_2} \mathds{1}_{x_2 = u}) (\sum_{x_1, y_1, y_2} \mathds{1}_{u = y_2 + (x_1 - y_1)^{-1}}). 
\end{align*}
Applying the Cauchy-Schwarz inequality, we see that
\begin{align*}
    H(X,Y) & \leq (\sum_{u} \sum_{x_1, x_2} \mathds{1}_{x_1 = x_2 = u})^{1/2} (\sum_{u} \sum_{x_1, y_1, y_2, x_3, y_3, y_4} \mathds{1}_{u = y_2 + (x_1 - y_1)^{-1} = y_3 + (x_3 - y_4)^{-1}})^{1/2} \\
  &  = |X|^{1/2} H_1(X,Y)^{1/2}, 
\end{align*}
where $H_1(X,Y)$ counts the number of solutions to the equation
\begin{equation} \label{cas3}
    y_2 + (x_1 - y_1)^{-1} = y_3 + (x_3 - y_4)^{-1} ,
\end{equation} 
with $x_1, x_3 \in X$ and $y_1, \dots, y_4 \in Y$. Thus, it suffices to show that
\[ H_1(X,Y) \ll |X|^{4/3} |Y|^{3}. \]
\par

We begin this endeavour by considering the solutions where $y_2 = y_3$, which can trivially be bounded above by $|X||Y|^3$, and so, it suffices to assume that $y_2 \neq y_3$. We denote $I_1$ to be the number of such solutions. Rewriting $\eqref{cas3}$ as
\[ x_1 = \frac{ x_3 (y_1  +(y_3 - y_2)^{-1} ) + (y_1 - y_4)(y_3 - y_2)^{-1} - y_1 y_4 }{x_3 +(y_3 - y_2)^{-1} - y_4}, \]
we see that $I_1$ counts the number of solutions to the equation
\[ x_1 = \frac{ x_3 (y_1  +d ) + (y_1 - y_4)d - y_1 y_4 }{x_3 +d - y_4} \ \ \text{with} \ d \neq 0, \]
where each solution $x_1, x_3,y_1,y_4,d$ is being counted with the weight $r(d^{-1})$, with $r(n) = |\{(y,y') \in Y \times Y \ | \ n = y - y'\}|$ for each $n \in \mathbb{R}$. Furthermore, setting 
\[ u_1 = y_1 + d \ \text{and} \ u_2 =  (y_1 - y_4)d - y_1 y_4 \ \text{and} \ u_3 = d - y_4, \]
we see that the preceding expression corresponds to the equation $x_1 = M_{\vec{u}}(x_3)$. Since $u_1 u_3 - u_2 = d^2 \neq 0$, we may deduce that each choice of $(u_1, u_2, u_3)$ corresponds to at most $2$ choices of $(y_1, y_4, d)$, which, in turn, allows us to provide upper bounds for $I_1$ in terms of weighted incidences between sets of points and M\"{o}bius transformations. The latter can then be estimated using Lemma $\ref{wtin}$, and in particular, we get that
\begin{align*}
 I_1 \ll  
 & \ |X|^{4/3} \big(\sum_{\vec{n} \in H} w(\vec{n})^2 \big)^{1/3} \big(\sum_{\vec{n} \in H} w(\vec{n}) \big)^{1/3} \ + \ \sup_{\vec{n} \in H} w(\vec{n}) |X|^2 \\
& +  |X|^{12/11} \big(\sum_{\vec{n} \in H} w(\vec{n})^2\big)^{2/11} \big(\sum_{\vec{n} \in H} w(\vec{n}) \big)^{7/11} \log |X| \ +  \sum_{\vec{n} \in H} w(\vec{n}),
\end{align*}
where $H = Y\times Y \times ((Y-Y)\setminus \{0\})^{-1}$ and $w(\vec{n}) = r(n_3^{-1})$. Using double counting, we see that
\[ \sum_{\vec{n} \in H} w(\vec{n}) = |Y|^4 \ \text{and} \  \sum_{\vec{n} \in H} w(\vec{n})^2 = |Y|^2 E_{2,2}(Y) \leq |Y|^5 \ \text{and} \sup_{\vec{n} \in H} w(\vec{n}) \leq |Y|,  \]
whence, 
\[ I_1 \ll |X|^{4/3} |Y|^3  + |X|^{12/11}  |Y|^{38/11} \log |X|+ |X|^2|Y| + |Y|^4 \ll |X|^{4/3} |Y|^3, \]
with the second inequality following from the fact that $|Y|^2 \leq |X| \leq |Y|^3$. Utilising this along with the bound $H_1(X,Y) \ll I_1 + |X| |Y|^3$ finishes the proof of Theorem $\ref{hyp}$.
\end{proof}

It appears to be that Theorem $\ref{hyp}$ provides the best known bounds for such hyperbolic incidences in the regime when $|Y|^2 \leq |X| \leq |Y|^3$, and we refer the reader to \cite{RW2021} for more details on the problem of estimating $H(X,Y)$ in various other settings.

We conclude this section by recording examples of sets $X,Y$ which provide large values of $H(X,Y)$.

\begin{Proposition} \label{1cons}
There exist arbitrarily large sets $X, Y \subset \mathbb{Q}$ such that $|X| \geq |Y|$ and 
\[H(X,Y) \gg |X||Y| e^{c \frac{\log |Y|}{ \log \log |Y|}}, \]
for some absolute constant $c>0$.
\end{Proposition}

\begin{proof} 
 We will show that there exist arbitrarily large sets $X, Y \subseteq \mathbb{N}$, some perfect square $n\geq 1$ and some absolute constant $c>0$ such that
\begin{equation} \label{eop}
     \sum_{x_1, x_2 \in X, y_1, y_2 \in  Y} \mathds{1}_{(x_1 - y_1)(x_2 - y_2) = n} \gg |X||Y| e^{c \frac{\log |Y|}{\log \log |Y|}},      
\end{equation} 
since the claimed result follows by considering the sets $n^{-1/2} \cdot X$ and $n^{-1/2} \cdot Y$. In this endeavour, we set $r$ to be some sufficiently large integer, $n = (p_1 \dots p_r)^{10}$ where $p_1, \dots, p_r$ are the first $r$ prime numbers and $S = \{ p_1^{e_1}\dots p_r^{e_r} : 0 \leq e_i \leq 10\}$. Observe that for any finite, non-empty set $Y \subseteq \mathbb{N}$ and any $X = Y + S$, we have that
\begin{equation} \label{j49}
 \sum_{x_1, x_2 \in X, y_1, y_2 \in  Y} \mathds{1}_{(x_1 - y_1)(x_2 - y_2) = n}  \gg |Y|^2 |S| \geq |X||Y|.
\end{equation} 
 We will now show that upon setting  $Y = \{1,2,\dots,n\}$, one can obtain a slightly better lower bound. Indeed, by the prime number theorem, we may deduce that
\[    n =  e^{10 \sum_{i=1}^r \log p_i}  \leq e^{C r \log r}, \]
for some absolute constant $C >0$, whence, $r  \gg \log n/ \log \log n.$ Combining this with the definition of $S$, we get that
\[  |S| \geq 10^r \geq  e^{c \log n/ \log \log n},  \]
for some absolute constant $c>0$. This may now be combined with \eqref{j49} and the fact that $n = |Y| \leq |X| \leq 2|Y|$ to obtain \eqref{eop}.
\end{proof}

We suspect that the above example indicates the optimal upper bound for $H(X,Y)$. It would be of independent interest to improve the above estimates for $H(X,Y)$, both by improving the power saving as well as extending the range $|X| \leq |Y|^3$ in Theorem \ref{hyp}. A combination of  Szemerédi--Trotter Theorem and Theorem \ref{hyp} implies that for any $\delta>0$ and any $X, Y \subseteq \mathbb{R}$ satisfying $|Y| \leq |X| \leq |Y|^{3 - \delta}$, one has $H(X,Y) \ll |X|^{1 - \delta/60}|Y|^2$. For the more general range where $|X| \leq |Y|^C$, with $C \geq 1$ being some absolute constant, one can follow Bourgain's ideas \cite{Bo2012} amalgamated with growth results in $\mathrm{SL}_2(\mathbb{C})$ to show that $H(X,Y) \ll_C |X|^{1 - c'} |Y|^2$, with a very weak dependence of $c'>0$ on $C>0$.


\section{Proofs of Theorems $\ref{th3}, \ref{th2}$ and $\ref{th1}$}

We dedicate this section to write the proofs of Theorems $\ref{th3}, \ref{th2}$ and $\ref{th1}$. We may assume that $|A| \geq 10$, since otherwise we are done by adjusting the implicit constant in the Vinogradov notation. Moreover, note that it suffices to prove the aforementioned results for sets $A$ of positive integers, since the equations $x_1 + \dots + x_s = x_{s+1} + \dots + x_{2s}$ and $x_1 \dots x_{s} = x_{s+1} \dots x_{2s}$ are dilation invariant and since
\[ \max\{|A \cap (0, \infty)|, |A \cap (-\infty, 0)| \} \gg |A|+1 .\]
We now present our proof of Theorem $\ref{th3}$.

\begin{proof}[Proof of Theorem $\ref{th3}$]
  We begin our proof by applying Lemma $\ref{mu1}$ to obtain the existence of disjoint sets $X,Y$ satisfying $A = X \cup Y$ and the fact that
\[ E_{s,2}(X) \ll_{s} |X|^{2s - \sigma_s} \ \text{and} \ M_{s,2}(Y) \ll_{s} |Y|^{2s - \sigma_s}, \]
for
\[ \sigma_s \geq D (\log \log s)^{1/2} (\log \log \log s)^{-1/2}. \]
Applying Lemmas $\ref{gens}$ and $\ref{mu2}$ for sets $X,Y$ respectively, we see that there exists a $B_{s}^+[1]$ set $B \subseteq X$ and a $B_{s}^\times[1]$ set $C \subseteq Y$ such that
\[ |B| \gg_{s}  |X|^{\frac{  \sigma_s + 1/s}{2s}   } \ \text{and} \ |C| \gg_{s}   |Y|^{\frac{  \sigma_s + 1/s}{2s}   } . \]
We obtain the desired conclusion by noting that $\max \{ |X|, |Y|\} \geq |A|/2$. 
\end{proof}

Next, we record our proof of Theorem $\ref{th2}$.

\begin{proof}[Proof of Theorem $\ref{th2}$]

Let $s \geq 3$, let $k=30s$ and let $c= 1/2s$. We divide our proof into two cases, and so, we first suppose that
\[ E_{s,k}(A) \leq |A|^{sk - k + 1/s - c}. \]
In this case, we may use Lemma $\ref{hsh}$ to deduce the existence of a $B_{s}^+[k-1]$ set $B$ with $|B| \gg_{s,k} |A|^{1/s + c'/sk}$ with $c'\geq\min\{c(k-2s)/k, 1/s\} \gg 1/s$, whereupon, we may assume that 
\[ \sum_{n \in sA} r_{s}(A;n)^k = E_{s,k}(A) > |A|^{sk - k + 1/s - c}. \]
This now implies that
\begin{equation} \label{ync3}
 \sup_{n \in sA} r_{s}(A;n) \geq (|A|^{sk - k + 1/s - c} |A|^{-s} )^{1/(k-1)} = |A|^{s  -1 - \nu},
 \end{equation}
where $\nu = (1 - 1/s + c)/(k-1)$. 
\par

We first deal with the case when $s \geq 4$. In this case, combining the first and the last inequality stated in Lemma $\ref{awk}$ along with the bound presented in $\eqref{ync3}$, we get that
\[ E_{2,2}(A) \geq |A|^{3 - \nu}, \]
in which case, we can apply Lemma $\ref{bsg5}$ to obtain $A' \subseteq A$ such that
\[ |A'| \gg |A|^{1 - \nu} \ \text{and} \ |A'-A'| \ll |A'|^{1 + 4\nu}. \]
Using Lemmata $\ref{prineq}$ and $\ref{so1}$, we may now infer that
\[ M_{2,2}(A') \ll |A'+A'|^2 \log |A'| \ll |A'|^{2 + 16 \nu} \log |A'|, \]
which, in turn, combines with Lemma $\ref{awk}$ to give us
\[ M_{s,2}(A') \ll |A'|^{2s - 2 + 16 \nu} \log |A'|. \]
By our choice of $k$, we have that $\log |A'| \ll_{s,k} |A'|^{\nu}$ and $17 \nu < 1/s$, and so, we can now employ Lemma $\ref{mu2}$ to obtain a  $B_{s}^\times[1]$ set $C \subseteq A'$ such that 
\[ |C| \gg |A'|^{1/s + (1/s - 17 \nu)/2s} \gg |A|^{(1/s + (1/s - 17 \nu)/2s)(1 - \nu)}=|A|^{1/s+\mu}, \]
where $\mu= (1/s - 17 \nu)( 1- \nu)/2s  - \nu/s$. 
As $c=1/2s$ and $k= 30s$, by elementary computations we have $\mu \gg 1/s^2 > 0$, which proves Theorem $\ref{th2}$ for $s\geq 4$. 


The $s=3$ case follows similarly, except this time, a straightforward application of the Cauchy-Schwarz inequality combined with $\eqref{ync3}$ implies that
\[ E_{2,2}(A) \geq|A|^{3 - \nu'}, \]
with $\nu' = 2 \nu$. In order to see this, note that
\[ r_{s}(A;n) = \sum_{a_1, a_2, a_3 \in A} \mathds{1}_{n - a_1 = a_2 + a_3} \leq |A|^{1/2} \big(\sum_{a_1 \in A} \big(\sum_{a_2, a_3 \in A} \mathds{1}_{n-a_1 = a_2 + a_3} \big)^2 \big)^{1/2} \leq |A|^{1/2} E_{2,2}(A)^{1/2}.     \]
With a lower bound for $E_{2,2}(A)$ in hand, we now proceed as in the setting when $s \geq 4$ to obtain a $B_3^\times [1]$ set $C\subseteq A$ with $|C|\gg |A|^{1/3+\mu}$, where 
\[ \mu=(1/3-34\nu)(1-2\nu)/6-2\nu/3. \]
As $k\geq 90$, we have $\mu>0$, which finishes the $s=3$ case of Theorem $\ref{th2}$. 
\end{proof}

Finally, we state the proof of Theorem $\ref{th1}$.

\begin{proof}[Proof of Theorem $\ref{th1}$]
Let $k = 32$ and $\eta = 1802/3630$ and let $\varepsilon = (1- \eta)(k-1)^{-1}$. We divide our proof into two cases, wherein, we first assume that the set 
\[ S = \{ x \in 2A \ | \ r_{2}(A; x) \geq |A|^{1 - \varepsilon} \} \]
satisfies $|S| \leq |A|^{\eta}$. In this case, we see that
\begin{align*}
 E_{2,k}(A)
 &  = \sum_{x \in 2A} r_{2}(A; x)^k = \sum_{x \in 2A \setminus S} r_{2}(A; x)^k + \sum_{x \in S} r_{2}(A; x)^k  \\
& \leq (|A|^{1 - \varepsilon})^{k-1}  \sum_{x \in 2A \setminus S} r_{2}(A; x) + |A|^k |S| \\
& \leq |A|^{k + 1 - \varepsilon(k-1)} + |A|^{k + \eta} \ll |A|^{k + \eta}.  
 \end{align*}
 We may now apply Lemma $\ref{sidr}$ to deduce the existence of some $B_{2}^+[k-1]$ set $B \subseteq A$ such that $|B| \gg_{k} |A|^{1/2 + \delta_1}$, where $\delta_1 = (1/2 - \eta)/2k$. Since $\eta < 1/2$, we have that $\delta_1 >0$, and consequently, we are done in this case.
 \par
 
Our second case is when $|S| > |A|^{\eta}$, whereupon, we choose a subset $S'$ of $S$ such that $|S'| = |A|^{\eta}$. Note that
\[ |A|^{1- \varepsilon} |S'| \leq \sum_{x \in S'} \sum_{a,b \in A} \mathds{1}_{a = x - b} = \sum_{a \in A} R(a), \]
where $R(n) = |\{(x,b) \in S' \times A \ | \ n = x- b \}|$ for each $n \in \mathbb{R}$. Writing
\[ A' = \{ a \in A \ | \ R(a) \geq |S'|/2|A|^{\varepsilon}\}, \]
we see that
\[ \sum_{a \in A \setminus A'} R(a) < |S'| |A|^{1- \varepsilon}/2, \]
which then combines with the preceding inequality to deliver the bound
\[ |S'||A'| \geq \sum_{a \in A'} R(a) = \sum_{a \in A} R(a) - \sum_{a \in A \setminus A'} R(a) > |A|^{1-\varepsilon} |S'|/2 . \]
Thus, we have that $|A'| \geq |A|^{1- \varepsilon}/2$. 
\par

We now claim that for every real number $n \neq 0$, we have that 
\[ m_{2}(A', n) \ll |A|^{1 + c + 2\varepsilon - 3c \eta},\]
where $c = 1/6$. In order to see this, we note that
\[ |A|^{2 \eta-2 \varepsilon} m_{2}(A', n) \ll \sum_{a_1, a_2 \in A'} R(a_1) R(a_2) \mathds{1}_{n = a_1 a_2} \leq H(n^{-1/2} \cdot A, n^{-1/2} \cdot S'), \]
where for any $\lambda \neq 0$ and any $X \subset \mathbb{R}$, we define $\lambda \cdot X = \{ \lambda \cdot x : x \in X\}$. Since $1/3 < \eta < 1/2$, we may use Theorem $\ref{hyp}$ to infer the bound
\[ H(n^{-1/2} \cdot A, n^{-1/2} \cdot S') \ll |A|^{1 + c} |S'|^{2 - 3c} = |A|^{1 + c + 2\eta - 3c \eta} . \]
These two inequalities combine together to give the claimed bound.
\par

With the above estimate in hand, we see that
\begin{align*}
M_{2,k}(A')
& = \sum_{n \in A'^{(2)}} m_{2}(A';n)^k \ll_k |A|^{(1 + c + 2\varepsilon - 3c \eta)(k-1)} \sum_{n \in A'^{(2)}} m_{2}(A';n) \\
& \ll_{k} |A|^{(1 + c + 2\varepsilon - 3c \eta)(k-1)} |A'|^2   \ll_{k} |A'|^{\frac{(1 + c + 2\varepsilon - 3c \eta)(k-1)}{(1- \varepsilon)} + 2} \\
& = |A'|^{k + 1/2 - \frac{(k- 3/2)(1-\varepsilon) - (1 + c + 2\varepsilon - 3c \eta)(k-1)}{(1- \varepsilon)} }.
\end{align*}
Thus, applying Lemma $\ref{sidr}$ again, we see that $A$ must have a $B_{2}^\times[k-1]$ set $C$ such that 
\[ |C| \gg_{k} |A'|^{1/2 + \frac{(k- 3/2)(1-\varepsilon) - (1 + c + 2\varepsilon - 3c \eta)(k-1)}{2k(1- \varepsilon)}} \gg_{k} |A|^{1/2 + \delta_2}, \]
where
\[ \delta_2 = \frac{(k- 3/2)(1-\varepsilon) - (1 + c + 2\varepsilon - 3c \eta)(k-1) - k \varepsilon}{2k}. \]
The reader may now verify that the facts that $c = 1/6$ and $\varepsilon = (1- \eta)/(k-1)$ and $k = 32$ and $\eta = 1802/3630$ allow us to have $\delta_2>0$, whence, we are done with the proof of Theorem $\ref{th3}$.
\end{proof}

We conclude this section by noting that our proofs of Theorems $\ref{th2}$ and $\ref{th3}$ essentially involve proving variants of the aforementioned low-energy decompositions. For instance, the following statement can be recovered from the proof of Theorem $\ref{th2}$, which may be of independent interest.

\begin{Proposition} \label{ghz}
Let $s,k \geq 4$ be natural numbers and let $c>0$ be a real number such that $\nu = (1- 1/s + c)/(k-1) \in (0, 1)$. Then given any finite set $A \subset \mathbb{Z}$, either we have that $E_{s,k}(A) \ll_{s,k} |A|^{sk - k + 1/s - c}$ or there exists $A' \subseteq A$ such that 
\[ |A'| \gg |A|^{1 - \nu} \ \text{and} \ M_{l,2}(A') \ll |A'|^{2l - 2 + 16 \nu} \log |A'|, \]
for every $l \geq 2$. 
\end{Proposition}

Similar results may be derived for the cases when $2 \leq s \leq 3$ from the above proofs, and we leave these details to the reader.



\bibliographystyle{amsbracket}
\providecommand{\bysame}{\leavevmode\hbox to3em{\hrulefill}\thinspace}

\end{document}